\documentclass[11pt]{article}
\usepackage{amssymb}
\usepackage{amsmath}
\usepackage{color}
\usepackage{amsthm}
\usepackage[latin1]{inputenc} 
\usepackage[T1]{fontenc}      
\usepackage{graphicx}
\usepackage{eurosym}
\usepackage{hyperref}
\newtheorem{theorem}{Theorem}

\newtheorem{lemma}[theorem]{Lemma}
\newtheorem{proposition}[theorem]{Proposition}
\newtheorem{remark}[theorem]{Remark}
\newtheorem{assumption}[theorem]{Assumption}

\newtheorem{example}[theorem]{Example}

\newcommand{\xabd}{x_{\alpha,\beta}^\delta}
\newcommand{\xbe}{x_{\beta,\eta}^\delta}

\newcommand{\xdag}{x^\dagger}
\newcommand{\yd}{y^\delta}
\newcommand{\range}{\rm range}

\newcommand{\R}{\mathbb{R}}
\newcommand{\N}{\mathbb{N}}
\newcommand{\1}{\ell^1}
\newcommand{\2}{\ell^2}
\newcommand{\3}{\ell^\infty}
\newcommand{\M}{\mathcal{M}}

\newcommand\domain[1]{{{\mathcal{D}}(#1)}}

\newcommand{\sgn}{\mathop{\mathrm{sgn}}}

\title{\bf Elastic-net regularization versus $\ell^1$-regularization  \\
for linear inverse problems with quasi-sparse solutions}

\begin{document}

\author{De-Han Chen\thanks{Department of
Mathematics, The Chinese University of Hong Kong, Shatin, N.T., Hong
Kong. ({\tt dhchen@math.cuhk.edu.hk})}
\and Bernd Hofmann\thanks{Technische Universit\"at Chemnitz, Fakult\"at f\"ur Mathematik,  09107 Chemnitz,
Germany.
The research work was initiated during the visit of this author to The Chinese University of
Hong Kong in March 2014 and supported partially by a Direct Grant for Research from The Chinese University of Hong Kong.
The work of this author was substantially supported by German Research Foundation (DFG) under grant HO~1454/8-2.
({\tt bernd.hofmann@mathematik.tu-chemnitz.de})}
\and Jun Zou\thanks{Department of
Mathematics, The Chinese University of Hong Kong, Shatin, N.T., Hong
Kong. The work of this author was substantially supported by Hong Kong RGC grants
(Projects 405513 and 14306814). ({\tt zou@math.cuhk.edu.hk})}}

\date{}

\maketitle

\begin{abstract}
We consider the ill-posed operator equation $Ax=y$ with an injective and bounded linear operator $A$ mapping between $\ell^2$ and a Hilbert space $Y$, possessing the unique solution \linebreak $\xdag=\{\xdag_k\}_{k=1}^\infty$.  For the cases that sparsity  $\xdag \in \ell^0$ is expected but often slightly violated in practice, we investigate in comparison with the $\ell^1$-regularization
the elastic-net regularization, where the penalty is a weighted superposition of the $\ell^1$-norm and the $\ell^2$-norm square, under the assumption that $\xdag \in \1$. There occur two positive parameters in this approach, the weight parameter $\eta$
and the regularization parameter as the multiplier of the whole penalty in the Tikhonov functional, whereas only one regularization parameter arises in $\1$-regularization. Based on the variational inequality
approach for the description of the solution smoothness with respect to the forward operator $A$ and exploiting the
method of approximate source conditions, we
 present some results to estimate the rate of convergence for the elastic-net regularization. The occurring rate function contains the rate of the decay $\xdag_k \to 0$ for $k \to \infty$ and the classical smoothness properties of $\xdag$ as an element in $\ell^2$.
\end{abstract}

\vspace{0.3cm}

{\parindent0em {\bf MSC2010 subject classification:}
65J20,} 47A52, 49J40

\vspace{0.3cm}

{\parindent0em {\bf Keywords:}
Linear ill-posed problems}, sparsity constraints,
elastic-net regularization,\\ $\ell^1$-regularization, convergence rates,
source conditions.

\section{Introduction}\label{s1}
\setcounter{equation}{0}
\setcounter{theorem}{0}
In this paper, we are interested in studying the linear ill-posed problem
\begin{equation}\label{eq:opeq}
Ax=y, \quad x\in \2, ~~y\in Y\,,
\end{equation}
where $Y$ is an infinite dimensional real Hilbert space, and $A:\2\to Y$ an injective and bounded linear operator with a non-closed range. Since $A^{-1}:\range( A)\subset Y\to$ $ \2$ is unbounded in this case, the corresponding system \eqref{eq:opeq}
suffers from ill-posedness in the sense that solutions may not exist if the exact data
$y=A\xdag$ for $\xdag \in \2$ comes with noise, namely
only the noisy data $y^\delta$ of $y$ is available, where $\delta>0$ represents the noisy level in the data, i.e.,
$\|y^\delta-y\|_Y\leq \delta$;
and when solutions exist they may still be far away from the exact solution $x^\dagger$ to (\ref{eq:opeq}), even if $\delta$ is small. In section~\ref{s2}, we will outline that general ill-posed linear operator equations
in Hilbert spaces can be rewritten into the form (\ref{eq:opeq}).

{The most widely adopted approach for regularizing the ill-posed system (\ref{eq:opeq}) is the Tikhonov regularization, which aims at
finding the approximate solutions $x_\gamma^\delta$ to problem (\ref{eq:opeq}) as the minimizers of the variational
problem
\begin{equation}\label{eq:Tikl1}
T_\gamma^\delta(x):=\frac{1}{2}\|Ax-y^\delta\|_Y^2+\gamma \mathcal{R}(x) \to \min,
\end{equation}
where $\gamma>0$ is the regularization parameter, and $\mathcal{R}(x)$ is the penalty that may be
chosen appropriately, with popular examples like
$\|x\|_{\2}^2$, $\|x\|_{TV}$, $\|x\|_{H}$ or $\|x\|^q_{\ell^q}:=\sum \limits_{k=1}^\infty |x_k|^q$
for $1\leq q<+\infty$. In particular,
it was shown in \cite{GrasHaltSch08} and \cite{Lorenz08} that the penalty functional $\|x\|_{\1}$ ensures that the $\1$-regularized solutions $x_\gamma^\delta$ to the variational problem
\begin{equation}\label{eq:Tikl10}
T_\gamma^\delta(x):=\frac{1}{2}\|Ax-y^\delta\|_Y^2+\gamma \|x\|_{\1} \to \min,\quad \mbox{subject to}\quad x \in \1,
\end{equation}
provide} stable approximate solutions to equation (\ref{eq:opeq}) if the exact solution
$\xdag=\{\xdag_k\}_{k=1}^\infty$ is sparse, i.e., $x_k \not=0$ occurs only for a finite number of components. The sparsity has been recognized as an important structure in many fields, e.g. geophysics~\cite{Taylor}, imaging science~\cite{Figueiredo07}, statistics~\cite{Tibshirani96} and signal processing~\cite{Candes}, and hence has received considerable attention. In this work, motivated by the recent works
on the multi-parameter Tikhonov functional~\cite{JZ12,JLS09,WLMC13}, we consider the following multi-parameter variational problem
\begin{equation} \label{eq:Tikelastic}
T_{\alpha,\beta}^\delta(x):= \frac{1}{2}\|Ax-y^\delta\|^2_Y +\alpha\,\|x\|_{\1}+ \frac{\beta}{2}\|x\|^2_{\2}\to \min, \quad \mbox{subject to}\quad x \in \1,
\end{equation}
which is called \emph{the elastic-net regularization}. The functional $T_{\alpha,\beta}^\delta$ was originally used in statistics~\cite{ZH05}. The major motivation is the observation that the $\1$-regularization fails to identify group structure for problems with highly correlated features, and tends to select only one feature out of a relevant group. It was proposed and confirmed numerically in \cite{ZH05} that the elastic-net regularization may retrieve the whole relevant group correctly. For an application of the elastic-net regularization to learning theory, one may refer to \cite{DeMol09}. Furthermore, the stability of the minimizer and its consistency have been studied, and convergence rates for both a priori and a posteriori parameter choice have been established under suitable source conditions (cf.~\cite{JLS09}). {Moreover, we would also like to emphasize that
the elastic-net regularization can be viewed as a special case of the $\1$-regularization,
due to the identity
$$\frac{1}{2}\|Ax-y^\delta\|_Y^2+\alpha\,\|x\|_{\1}+ \frac{\beta}{2}\|x\|^2_{\2}= \frac{1}{2}\left\|\left[\begin{array}{c}A\\\sqrt{\beta} I\end{array} \right]x-\left[\begin{array}{c} y^\delta\\0\end{array} \right]\right\|_{Y \times \2}^2+\alpha\,\|x\|_{\1}.$$
}

As it was done for the {\sl $\1$-regularization} (\ref{eq:Tikl10}) in \cite{BFH13}, we intend to enrich with the present work the analysis on elastic-net regularization by taking into account the case that the solution $\xdag$ is not truly sparse in many applications,
but has infinitely many
nonzero components $\xdag_k$ that decay sufficiently rapidly to zero as $k \to \infty$. We shall call this kind of solutions
to be {\it quasi-sparse} in the sequel for convenience, which occur often in practice,
e.g., {when applying wavelets to audio signals or natural images (cf.~\cite{HHO97,WCP92}), where
the compression algorithms are usually constructed by making use of the fact that most coefficients are very small and can be ignored.}
We shall model the quasi-sparse solutions with the assumption $\xdag \in \1$.
Following \cite{JLS09}, we consider for elastic-net regularization the pair $(\beta,\eta)$ of positive regularization parameters instead of the pair $(\alpha,\beta)$ by setting $\eta:=\alpha/\beta$, then
(\ref{eq:Tikelastic}) is reformulated as
 \begin{equation} \label{eq:Tiketa}
T_{\beta,\eta}^\delta(x):=\frac{1}{2}\|Ax-y^\delta\|^2_Y +\beta\,\mathcal{R}_\eta(x) \to \min, \quad \mbox{subject to}\quad x \in X,
\end{equation}
with the penalty functional
\begin{equation} \label{eq:peneta}
\mathcal{R}_\eta(x):=\eta\,\|x\|_{\1}+\frac{1}{2}\|x\|^2_{\2},
\end{equation}
and we denote by $\xbe$ the minimizers to (\ref{eq:Tiketa}).
The degenerate form of $\ell^1$-regularization, i.e., $\beta=0$ in (\ref{eq:Tikelastic}),  was studied intensively including convergence rates in \cite{BFH13} (see also the extensions in \cite{FHV16} and references therein) for the cases with quasi-sparse solutions.
We will extend the results from
\cite{BFH13},  \cite{Hof06} and \cite{JLS09} to analyze the two-parameter situation of elastic-net regularization with respect to convergence rates when the sparsity assumption fails. {It is worth mentioning  that other modifications of (\ref{eq:Tikl10}) have already been discussed in literature, for instance, the term  $\gamma\,\|x\|_{\1}$ in the penalty functional may be
replaced by some weighted or modified versions (cf., e.g.,~\cite{Lorenz08,RamRes10}), or
alternatively by non-convex sparsity-promoting terms like $\gamma\,\|x\|_{\ell^q}$ for $0<q<1$  (cf., e.g.,~\cite{BreLor09,Zarzer09}) or
$\gamma\,\|x\|_{\ell^0}:=\gamma\,\sum\limits_{k=1}^\infty \sgn(|x_k|)$ (cf.~\cite{WLMC13}).}
However, the theory with respect to convergence rates for the cases with quasi-sparse solutions
are still rather limited compared with the case of truly sparse solutions.

The paper is organized as follows. In section 2 we will fix the basic problem setup, notations and assumptions, and then proceed to
 an overview of the smoothness conditions for proving convergence rates of single or multi-parameter Tikhonov regularization,
 where we shall show that the source conditions  do not hold when the sparsity is violated.  In section \ref{s3} we derive the  convergence rates
 of regularized solutions for general linear ill-posed problems under variational inequalities, in which the
 regularization parameter is chosen according to three varieties of a posteriori parameter choices, i.e., two-sided discrepancy
 principle, sequential discrepancy principle and adapted  Lepski{\u \i} principle.
 These results are then applied to
 $\1$-regularization  (\ref{eq:Tikl1}) directly in section \ref{s4}.  In section \ref{s5},
 by deriving an appropriate variational inequality, we establish the convergence rates for the regularized solutions $\xbe$ of elastic-net regularization  (\ref{eq:Tikelastic}) for
 a fixed $\eta>0$ and the case with quasi-sparse solutions.

\section{Problem setting and basic assumptions}\label{s2}
\setcounter{equation}{0}
\setcounter{theorem}{0}
Let $\widetilde X$ (resp.~$Y$) be  an infinite dimensional real Hilbert space, endowed with an inner product $\langle \cdot,\cdot\rangle_{\widetilde X}$ (resp.~$\langle \cdot,\cdot\rangle_{Y}$) and a norm $\|\cdot\|_{\widetilde X}$ (resp.~$\|\cdot\|_Y$),
 $\widetilde X$ be  separable, and  $\widetilde A \in \mathcal{L}(\widetilde X,Y)$ an injective and bounded linear operator mapping between $\widetilde X$
and $Y$.
In addition, we assume that $\range(\widetilde A)$ of $\widetilde A$ is not closed,
which is equivalent to that the inverse
${\widetilde A}^{-1}: \range(\widetilde A) \subset Y \to \widetilde X$ is unbounded. Thus
the operator equation
\begin{equation} \label{eq:tildeeq}
\widetilde A\,\widetilde x\,=\,y, \qquad \widetilde x \in \widetilde X,\quad y \in Y,
\end{equation}
with uniquely determined solution $\widetilde \xdag \in \widetilde X$ is ill-posed. This means that for noisy data $\yd \in Y$ replacing $y \in \range(\widetilde A)$  in (\ref{eq:tildeeq}), solutions may not exist, and even
when they exist the solutions may be still far from $\widetilde \xdag$
under the deterministic noise model
\begin{equation}\label{eq:noise}
\|y-y^\delta\|_Y \le \delta,
\end{equation}
with small noise level $\delta>0$.


{With the setting $A:=\widetilde A \circ U$, where $U:\2 \to \widetilde X$ is the unitary synthesis operator characterizing the isometric isomorphy between the separable Hilbert spaces $\widetilde X$ and $\2$, the operator equation (\ref{eq:tildeeq})
can be rewritten in the form
\begin{equation}\label{eq:opeq0}
 A\,x\,=\,y, \qquad x \in \2,\quad y \in Y.
\end{equation}
This transforms \eqref{eq:tildeeq} to the desired structure (\ref{eq:opeq}).
We note that $A$ is also injective, so this linear operator equation is ill-posed,
i.e.,~$\range(A)$ $ \not= \overline {\range(A)}^{\,Y}$.}

For any sequence $x=\{x_k\}_{k=1}^\infty$, we will denote by $\|x\|_{\ell^q}:=\left( \sum \limits_{k=1}^\infty |x_k|^q\right)^{1/q}$ the norm in the Banach spaces $\ell^q$ for $1 \le q<\infty$, and by $\|x\|_{\ell^\infty}:=\sup \limits_{k \in \N} |x_k|$  the norm in  $\ell^\infty$. The same norm  $\|x\|_{c_0}:=\sup \limits_{k \in \N} |x_k|$
is used for the Banach space $c_0$ of infinite sequences tending to zero. On the other hand, the symbol $\ell^0$ will stand for the set of all sparse sequences $x$, where $x_k \not=0$ occurs
only for a finite number of components. In the sequel we also set for short
$$X:=\1\,, $$
and consequently $X^*$ for the dual space $\3$ of $X$.

In the sequel, let $\langle \cdot,\cdot\rangle_{B^*\times B}$ denote the dual pairing between a Banach space $B$ and
its dual space $B^*$, and  $v_n \rightharpoonup v_0$  stand for the weak convergence in $B$, i.e.,
$\lim \limits_{n \to \infty}\langle w,v_n\rangle_{B^*\times B}=\langle w,v_0\rangle_{B^*\times B}$   for all $w \in B^*$.
For a Hilbert space $B$ we identify $B$ and $B^*$ such that weak convergence takes the form $\lim \limits_{n \to \infty}\langle w,v_n\rangle_{B}=\langle w,v_0\rangle_{B}$   for all $w \in B$.
Furthermore, we denote by  $e^{(k)}$, with $1$ at the $k$th position for $k=1,2,...$,  the elements of the standard orthonormal basis in $\ell^2$, which also is the normalized canonical Schauder basis in  $c_0$ and $\ell^q$ ($1 \le q<\infty$). That is, we find $\lim \limits_{n \to \infty}\|x-\sum \limits_{k=1}^n x_k e^{(k)}\|_{c_0}=0$ for all $x \in c_0$ and $\lim \limits_{n \to \infty}\|x-\sum \limits_{k=1}^n x_k e^{(k)}\|_{\ell^q}=0$ for all $x \in  \ell^q$, $1 \le q<\infty$.
For the operator $A: \2 \to Y$ we can consider its adjoint operator $A^*: Y \to \2$ by the condition
$$\langle v,Ax \rangle_Y=\langle A^*v,x \rangle_{\2}   \qquad \mbox{for all} \quad x \in \2,\;v \in Y.$$

Now we are stating a set of assumptions for the further consideration of equation (\ref{eq:opeq})
with a uniquely determined solution $\xdag$ and of the regularized solutions $\xabd$ and $x_\gamma^\delta$ solving the
extremal problems (\ref{eq:Tikl10}) and (\ref{eq:Tikelastic}), respectively.

\begin{assumption} \label{ass:basic}
\begin{itemize} \item[]
\item[(a)] The operator $A$ in equation (\ref{eq:opeq}) is an injective and bounded linear operator mapping $\ell^2$ to the Hilbert space $Y$ with a non-closed range. i.e., $\range(A) \not= \overline {\range(A)}^{\,Y}$.

\item[(b)]  Element $\xdag\in \1$ solves equation  (\ref{eq:opeq}).

\item[(c)] {For each $k\in \N$, there exists $f^{(k)} \in Y$ such that
$e^{(k)}= A^*f^{(k)}$, i.e., it holds that \linebreak
 $x_k=\langle e^{(k)},x\rangle_{\2}=\langle f^{(k)},Ax\rangle_{Y}\,$ for all $x=\{x_k\}_{k=1}^\infty \in \2$.}
\end{itemize}
\end{assumption}

\begin{remark} \label{rem:ass} 

{{\rm Item (c) above seems to be only a technical condition. If one considers the general operator equation
(\ref{eq:tildeeq}),  then it is equivalent to that
$u^{(k)}=\widetilde A^* f^{(k)}, \;k \in \N,$ for all elements of the orthonormal basis $\{u^{(k)}\}_{k=1}^\infty$
in $\widetilde X$ characterizing the unitary operator $U$.
However, this condition was motivated for a wide class of linear inverse problems by using the Gelfand triple
\cite{AHR13}.
This series of independent source conditions for all $e^{(k)}$ is nothing but a requirement on the choice of the
basis elements $u^{(k)}$ in $\widetilde X$. Roughly speaking, the basis elements must be in some sense `smooth enough' under the auspices of the operator
$\widetilde A$.
}} \hfill\fbox{}
\end{remark}

\begin{proposition} \label{pro:c0}
The range  of $A^*: Y \to \2$ is a nonclosed subset of $\2$ but dense in the sense of
the $\2$-norm, i.e.,~$\overline {\range(A^*)}^{\,\2}=\2$. On the other hand, $\range(A^*)$ is always a subset of $c_0$ and hence not dense in $\3$ in the sense of the supremum norm, i.e.,~$\overline {\range(A^*)}^{\,\3}\not=\3$.
\end{proposition}
\begin{proof}
{The proof is based on the properties of $A$ from Assumption~\ref{ass:basic}(a).
Noting that $\2 \subset c_0$, we know
$\range(A^*)\subset c_0$ for the adjoint operator $A^*: Y \to \2$.
On the other hand, the condition $\overline {\range(A^*)}^{\,\2}=\2$ is a consequence of the injectivity of $A$, while the non-closedness of $\range(A^*)$ in $Y$ follows  from the closed range theorem (cf.~\cite{Yos95}).
}\end{proof}

The smoothness of the solution to the ill-posed operator equation (\ref{eq:opeq})
with respect to the
forward operator $A$ plays an important role for obtaining error estimates and convergence rates in
Tikhonov-type regularization, e.g., see \cite{Scherzetal09,SKHK12}.
Such smoothness can be expressed by source conditions. {In particular, for the most prominent form
of the $\ell^2$-regularization
\begin{equation} \label{eq:Tikquad}
 \frac{1}{2}\|Ax-y^\delta\|^2_Y+\gamma\,\|x\|^2_{\2} \to \min, \quad \mbox{subject to}\quad x \in \2\,,
 \end{equation}
the classical theory of the Tikhonov regularization in Hilbert spaces applies (cf.~\cite{EHN96,Groe84}).
By making use of the purely quadratic penalty,
the minimizers $x^\delta_\gamma$ achieve the convergence rate
\begin{equation} \label{eq:Hilbertrate}
\|x_\gamma^\delta-\xdag\|_{\2}=\mathcal{O}(\sqrt{\delta}) \quad \mbox{as} \quad \delta \to 0
\end{equation}
under the source condition that
\begin{equation} \label{eq:Hilbertsc}
\xdag= A^*v, \quad v \in Y,
\end{equation}
when the regularization parameter is chosen a priori as $\gamma=\gamma(\delta) \sim \delta$ or a posteriori as $\gamma=\gamma(\delta,y^\delta)$
based on the discrepancy principle $\|Ax_\gamma^\delta-y ^\delta\|_Y=\tau\,\delta$ for some prescribed $\tau \ge 1$.}

If $\xdag$ is not smooth enough to satisfy (\ref{eq:Hilbertsc}), then the method of {\it approximate source conditions}  may help to bridge this gap when the concave and nonincreasing distance function
\begin{equation} \label{eq:distfct}
d^{A^*}_{\xdag}(R):=\inf \limits _{v \in Y: \,\|v\|_{Y} \le R}\|\xdag-A^*v\|_{\2}, \qquad R>0,
\end{equation}
tends to zero as $R\to \infty$. In such  case the decay rate of $d^{A^*}_{\xdag}(R) \to 0$
as $R \to \infty$ characterizes the degree of violation with respect to (\ref{eq:Hilbertsc}). Then
the convergence rate is slower than (\ref{eq:Hilbertrate}) and the rate function depends on $d^{A^*}_{\xdag}$ (cf.~\cite{Hof06,HofMat07}, and also \cite{BakuKok04}). From Proposition~\ref{pro:c0} we have
for all elements $\xdag \notin \2 \setminus \range(A^*)$ that $\lim \limits_{R \to \infty} d^{A^*}_{\xdag}(R)=0$,
because $\overline {\range(A^*)}^{\,\2}=\2$.

{If the element $\xdag \in \2$ fails to satisfy the source condition (\ref{eq:Hilbertsc}), however, we know that there is an {\it index function} $g$ and a source element $w \in \2$ such that (cf.~\cite{MatHof08})
\begin{equation} \label{eq:how}
\xdag= g(A^*A)\, w
\end{equation}
Here we say that
$g: (0,\infty) \to (0,\infty)$ is an index function if it is continuous, strictly increasing and
satisfies the  condition $\lim \limits _{t \to 0^+} g(t)=0$.
For $g(t)=\sqrt{t}$, the source conditions (\ref{eq:Hilbertsc}) and (\ref{eq:how}) are equivalent.
For any $\xdag\in \ell^2$,  the connection between $g$ in (\ref{eq:how}) and the distance functions in (\ref{eq:distfct}) were
outlined in \cite[\S~5.3]{HofMat07}.} In particular, we can find from \cite[Theorem 3.1]{DHY07} that,
for any exponents $0<\theta<1$,
\begin{equation} \label{eq:powerdist}
d^{A^*}_{\xdag}(R) \le \frac{K}{R^\frac{\theta}{1-\theta} }\quad \forall \, R\geq \underline{R} \quad \mbox{if} \quad \xdag \in \range[(A^*A)^{\theta/2}],
\end{equation}
with some positive constants $K$ and $\underline{R}$. In this case we say that $\xdag$ satisfies a H\"{o}lder source condition with exponent $\theta$.

For a general convex but not purely quadratic penalty functional $\mathcal{R}$,  the benchmark source condition is given by
\begin{equation} \label{eq:Banachsc}
\xi^\dagger=A^*v, \quad v\in Y\,,
\end{equation}
for some subgradient $\xi^\dagger \in \partial\mathcal{R}(\xdag)$ and source element $v\in Y$ (cf.~\cite{BurOsh04}).
As a result of the following proposition,  the source condition (\ref{eq:Banachsc}) for $\mathcal{R}=\mathcal{R}_\eta$ (see (\ref{eq:peneta})) can only hold if the solution is sparse, i.e.,~$\xdag \in \ell^0$. Hence this condition, which is also important for the convergence rates of elastic-net regularization in \cite{JLS09}, completely fails for the quasi-sparsity case of our interest
in this work.
The approximate source condition approach was extended to the general Banach space situation with convex penalties $\mathcal{R}$ in \cite{HeinHof09}. As an analog to (\ref{eq:distfct}), the corresponding distance functions for $\mathcal{R}=\mathcal{R}_\eta$ attain the form
\begin{equation} \label{eq:distxi}
d_{\xi^\dagger}(R):=\inf \limits _{v \in Y: \,\|v\|_{Y} \le R}\|\xi^\dagger-A^*v\|_{\3}, \qquad R>0.
\end{equation}
However, the proposition below also implies that this approximate source condition also fails, i.e.,
 $d_{\xi^\dagger}(R) \to 0$ as $R \to \infty$  cannot hold if $\xdag \in \1 \setminus \ell^0$,
which makes it impossible to verify convergence rates based on this approach.

On the other hand, explicit convergence rates of regularized solutions $\xabd$
for elastic-net Tikhonov regularization  (\ref{eq:Tikelastic}) may require
smoothness properties of $\xdag$ (cf.~\cite{ITO11}).  More precisely,
if for any $t\in[0,1]$ there exists $w_t$ such that
\begin{equation}\label{eq:multisource}
\xi_t:=A^* w_t\in
\partial \psi_{t}(\xdag),
\end{equation}
where
$ \psi_{t}(x):=t\|x\|_{\1}+(1-t)\|x\|^2_{\2}$, then the convergence rate
$$
D_{\xi_{t}}(x_{\alpha^{*}(\delta),\beta^{*}(\delta)}^\delta,\xdag)=\mathcal{O}(\delta)
$$
can be established, for Bregman distance $D_{\xi_{t}}(x,\xdag):=\psi_{t}(x)-\psi_{t}(\xdag)-\langle \xi_{t},x-\xdag\rangle_{\2}$, and parameter choice
$(\alpha^{*}(\delta),\beta^{*}(\delta))$ based on the multi-parameter discrepancy principle, i.e., $(\alpha^{*}(\delta),\beta^{*}(\delta))$ satisfies
$$
\|x_{\alpha^{*}(\delta),\beta^{*}(\delta)}-\xdag\|_{l^2}=c_m \delta^2
$$ with some prescribed  constant $c_m\geq 1$.
The proposition below also implies that the condition (\ref{eq:multisource})
fails if $\xdag$ is not truly sparse, because $\psi_t(x)=2(1-t)\mathcal{R}_{t/(2-2t)}(x)$ for all $t\in [0,1)$.

For each $\eta>0$ the convex functional $\mathcal{R}_\eta$ defined in (\ref{eq:peneta}) attains finite values on $X$. Moreover for each $x \in X$, by using the subgradients $\zeta=\{\zeta_k\}_{k=1}^\infty \subset X^*$ of $\|x\|_{\1}$, the subdifferential $\partial\mathcal{R}_\eta(x)$
collects all subgradients $\xi=\{\xi_k\}_{k=1}^\infty \subset X^*$ of the form
\begin{equation}\label{eq:subgr}
 \xi_k \, = \eta\,\zeta_k+x_k,\quad \mbox{where}\quad \zeta_k  \begin{cases} \quad =1 \quad \; \qquad \quad \mbox{if} \qquad x_k>0,\\ \in [-1,1] \qquad\quad \,\mbox{if}  \qquad x_k=0, \\\;\;= -1 \qquad \quad \quad \mbox{if} \qquad x_k<0, \end{cases} \qquad k \in \N.
 \end{equation}

\begin{proposition} \label{pro:xi}
If $\xdag \in \1 \setminus \ell^0$ and $v \in Y$, then for any $\eta>0$ the condition $A^*v \in \partial \mathcal{R}_\eta(\xdag)$ cannot hold. Also it does not hold that $d_{\xi^\dagger}(R) \to 0$ as $R \to \infty$ (see (\ref{eq:distxi})) in this case.
\end{proposition}

\begin{proof}
Assume that $A^*v \in \partial \mathcal{R}_\eta(\xdag)$ holds for some $\xdag \in \1 \setminus \ell^0$ and $v\in Y$.  Then, by formula (\ref{eq:subgr}) we have for every $\xi^\dagger:=A^*v \in \partial \mathcal{R}_\eta(\xdag)$ that
$\xi^\dagger_{k}=\pm \eta\zeta_k+\xdag_{k}$ for all $k \in \N$.
Then using the signum function
$$\sgn(z):=\begin{cases}\, \;\;1  \;\; \mbox{if} \;\; z>0,\\\;\; \,0 \; \;\mbox{if} \;\;z=0, \\-1 \; \;\mbox{if}\;\;z<0, \end{cases}
$$ there is a subsequence $\{\xdag_{k_l}\}_{l=1}^\infty$ of $\{\xdag_k\}^\infty_{k=1}$ such that
$|\sgn(\xdag_{k_l})|=1$ for all $l\in \N$ and $\lim \limits_{l \to \infty}|\xdag_{k_l}|=0$.
Therefore  we have
$$\left|\frac{[A^*v]_{k_l}- \xdag_{k_l}}{\eta}\right|=1,$$
which  gives a contradiction, because Proposition~\ref{pro:c0} implies  $A^*v \in c_0$, and hence that the left-hand side of this equation tends to zero for $l \to \infty$.  The second assertion is
a simple consequence of the fact that $\|\xi^\dagger-w\|_{\3}\ge \eta$ holds  for all elements $w \in c_0$ in this case.
\end{proof}

From the reasoning above we know that the source conditions always fail for the quasi-sparse solutions. To overcome the difficulty, we shall use the variational inequalities (variational source conditions) instead.
For more discussions about smoothness of solutions and their
 influences on convergence rates, we refer to \cite{HKPS07,HofMat12} and \cite{Flemmingbuch12,Fle12,Grasm10}
for further details.

\section{Convergence rates under variational inequalities}\label{s3}
\setcounter{equation}{0}
\setcounter{theorem}{0}
Throughout this section we extend our consideration to
 the more general situation of  an ill-posed operator equation
\begin{equation} \label{eq:genopeq}
A\,x\,=\,y, \qquad x \in Z,\quad y \in Y,
\end{equation}
and regularized
solutions $x_\gamma^\delta$ to (\ref{eq:genopeq}) with regularization parameter $\gamma>0$, which are minimizers of the functional
\begin{equation} \label{eq:genTik}
T_{\gamma}^\delta(x):=\frac{1}{2}\|Ax-y^\delta\|^2_Y+\gamma\,\mathcal{R}(x) \to \min, \quad \mbox{subject to}\quad x \in Z,
\end{equation}
where the nonnegative penalty functional $\mathcal{R}$ is  convex, lower semi-continuous and stabilizing,  and  $Z$
is a Hilbert space.  Here we call $\mathcal{R}$ stabilizing if the sublevel sets
$\M_c:=\{x \in Z:\,\mathcal{R}(x) \le c\}$ are weakly sequentially compact subsets of $Z$ for all $c \ge 0$.
In this context, the linear forward operator $A: Z \to Y$ is assumed to be injective and bounded,  and the uniquely determined solution $\xdag$ of (\ref{eq:genopeq})
is required to satisfy the condition that  $\xdag \in \domain{\mathcal{R}}$ with
$$\domain{\mathcal{R}}:=\{x \in Z:\,\mathcal{R}(x)<\infty\}.$$
Under the aforementioned assumptions  on $\mathcal{R}$, we know that  regularized solutions $x_\gamma^\delta$ exist
for all $\gamma>0$ and $\yd \in Y$,
and are stable with respect to perturbations in the data $\yd$ (cf., e.g., \cite{HKPS07,Scherzetal09,SKHK12}).

{For an index function $g$ and an error measure $E: \domain{\mathcal{R}} \times \domain{\mathcal{R}} \to \R_+$,
we need an appropriate choice $\gamma_*=\gamma_*(\delta,\yd)$ of the regularization parameter and
certain smoothness of $\xdag$ with respect to the forward operator $A$ to obtain
a convergence rate of the form
\begin{equation} \label{eq:concaverate}
E(x^\delta_{\gamma_*},\xdag)=\mathcal{O}(g(\delta)) \quad \mbox{as} \quad \delta \to 0\,.
\end{equation}
In particular, the variational inequalities of the form
\begin{equation}\label{eq:vi}
\lambda\,E(x,\xdag) \le \mathcal{R}(x)-\mathcal{R}(\xdag)+ C\,g\left(\|A(x-\xdag)\|_Y\right)  \quad \mbox{for all} \quad x \in \domain{\mathcal{R}}
\end{equation}
become popular for the description of solution smoothness, which were developed independently
in \cite{Flemmingbuch12} and \cite{Grasm10}, where $\lambda$ and $C$ are
constants satisfying $0<\lambda \le 1$ and $C>0$, and the index function $g$ was assumed to be concave.}

%

Concerning appropriate selection strategies for regularization parameters, we list  the following three
principles for a posteriori parameter choice $\gamma_*=\gamma(\delta,\yd)$,
for which Proposition~\ref{pro:general} below will apply and yield
corresponding reasonable convergence rates.

\begin{itemize}
\item[TDP:] For the prescribed $\tau_1$ and $\tau_2$ satisfying $1 \le \tau_1 \le \tau_2<\infty$,
the two-sided discrepancy principle (TDP) suggests to choose
the regularization parameter $\gamma_*=\gamma_{\scriptscriptstyle TDP}$ such that
\begin{equation}\label{eq:tdp}
\tau_1\,\delta \le \|Ax^\delta_{\gamma_{\scriptscriptstyle TDP}}-\yd\|_Y \le \tau_2\,\delta.
\end{equation}
Since the discrepancy functional $\|Ax^\delta_\gamma-\yd\|_Y$ is continuous and increasing with respect to
$\gamma \in (0,\infty)$, the regularization parameter $\gamma_{\scriptscriptstyle TDP}>0$ exists
for all $\yd \in Y$ whenever $\delta>0$ is sufficiently small. We refer to, e.g., \cite{AR10} for more details.
\item[SDP:] For the prescribed $\tau$, $q$ and $\gamma_0$
satisfying  $\tau >1,$  $0 < q < 1$ and $\gamma_0>0$, and the decreasing geometric sequence
$$ \Delta_q := \{ \gamma_j : \;\gamma_j=q^j \gamma_0, \; j \in \mathbb{N}\}, $$
the sequential discrepancy principle (SDP) suggests to choose
the regularization parameter $\gamma_*=\gamma_{\scriptscriptstyle SDP}$ such that
$\gamma_{\scriptscriptstyle SDP} \in \Delta_q$ satisfies
\begin{equation}\label{eq:sdp}
    \|A x_{\gamma_{\scriptscriptstyle SDP}}^\delta - \yd\|_Y \leq \tau \delta < \|A x_{\gamma_{\scriptscriptstyle SDP}/q}^{\delta} - \yd\|_Y.
\end{equation}
When using the SDP, we are
interested in finding the largest value $\gamma$ from the sequence $\Delta_q$ such that
$\|A x^\delta_\gamma-\yd\|_Y \le \tau \delta$. For the well-definedness of $\gamma_{\scriptscriptstyle SDP}$ from SDP,
its properties and convergence of regularized solutions $x^\delta_{\gamma_{\scriptscriptstyle SDP}}$ as $\delta \to 0$, we refer to \cite{AnzHofMat13}.
In principle, one can say that $\gamma_{\scriptscriptstyle SDP}$
is uniquely determined for all $0<q<1$ and $\yd \in Y$ whenever $\gamma_0>0$ is large enough.

\item[LEP:]  To apply the Lepski{\u \i} principle (LEP) for choosing the regularization parameter $\gamma>0$ under (\ref{eq:vi}), we restrict our consideration to the symmetric error measures $E$ satisfying the triangle inequality up to some constant $1\le C_E< \infty$, i.e., for all $x^{(i)} \in \domain{\mathcal{R}},\;i=1,2,3,$
\begin{equation} \label{eq:E1Lep}
E(x^{(1)},x^{(2)})=E(x^{(2)},x^{(1)})
\end{equation}
and
\begin{equation} \label{eq:E2Lep}
E(x^{(1)},x^{(2)}) \le C_E\,\left(E(x^{(1)},x^{(3)})+E(x^{(3)},x^{(2)})\right).
\end{equation}
For such symmetric error measures $E$,  the prescribed $q$ and $\gamma_0$ satisfying $q\in (0, 1)$ and $\gamma_0>0$,
 the increasing geometric sequence
$$
\widetilde \Delta_q := \{ \gamma_j : \;\gamma_j=\gamma_0/q^j, \; j \in \mathbb{N} \},
$$
and the strictly decreasing function
$$\Theta(\gamma):=\frac{17 \delta^2}{2\lambda\gamma}$$ {
 with a fixed $\delta>0$  and $\lambda$ from the variational inequality (\ref{eq:vi}),   which characterizes an upper bound of $E(x_\gamma^\delta,\xdag)$ for all $\gamma_0 \le \gamma \le \gamma_{apri}$ and a priori parameter choice $\gamma_{apri}=\frac{\delta^2}{C\,g(\delta)}$},  the adapted Lepski{\u \i} principle (LEP) suggests to choose
the regularization parameter $\gamma_*=\gamma_{\scriptscriptstyle LEP}$ such that
$\gamma_*$ is the largest value $\gamma$ in $\widetilde \Delta_q$  satisfying
\begin{equation}\label{eq:Lepski}
    E(x^\delta_{\gamma^\prime},x^\delta_{\gamma}) \le 2\,C_E\,\Theta(\gamma^\prime) \qquad \mbox{for all} \quad \gamma^\prime \in \widetilde \Delta_q \quad \mbox{with}\quad \gamma_0 \le \gamma^\prime<\gamma\,.
\end{equation}
We like to mention that the
LEP is based on a priori parameter choice (cf.~\cite[\S~4.2.1]{HofMat12} and
\cite[\S~1.1.5]{LuPer13}).
Under the variational inequality (\ref{eq:vi}) with an error measure $E$,
a priori parameter choice { $\gamma_{apri}$} yields the convergence rate (\ref{eq:concaverate}) for $\gamma_*=\gamma_{apri}$ (cf.~\cite[\S~4.1]{HofMat12}).

For the adapted Lepski{\u \i} principle, the following error estimate holds.
\begin{lemma} \label{lem:Lep}
{
Assume the variational inequality (\ref{eq:vi}) holds for a nonnegative error measure $E(\cdot)$ satisfying (\ref{eq:E1Lep}) and \eqref{eq:E2Lep}.
}
If $\gamma_0$ is sufficiently small such that $E(x_{\gamma_0}^\delta,\xdag) \le \Theta(\gamma_0)$,
then $\gamma_{\scriptscriptstyle LEP}$ from LEP is uniquely determined for all $0<q<1$ and $\yd \in Y$ and
meets the error estimate
\begin{equation}\label{eq:estilem}
E(x_{\gamma_{\scriptscriptstyle LEP}}^\delta,\xdag) \le C\, C_E(2+C_E)\,\frac{17}{2q\lambda}\,g(\delta),
\end{equation}
{ where the positive constants $C$ and $\lambda$ are from  (\ref{eq:vi}), and $C_E$ from (\ref{eq:E2Lep}). }
\end{lemma}
\begin{proof}
The proof goes along the same line as the one for Theorem~3 in \cite{HofMat12},
in combination with Proposition~1 from \cite{Mat06}. First we introduce
$$\gamma_+:=\max\{\gamma \in \widetilde \Delta:\,E(x_{\gamma^\prime}^\delta,\xdag) \le \Theta(\gamma^\prime) \quad \mbox{for all} \quad \gamma^\prime \in \widetilde \Delta_q,\;\gamma_0 \le \gamma^\prime \le \gamma\}.$$
Then we will show that $\gamma_k<\gamma_{apri} \le \gamma_k/q$ for some $k \in \N$ and
\begin{equation} \label{eq:help}
E(x_{\gamma_{\scriptscriptstyle LEP}}^\delta,\xdag) \le C_E(2+C_E) \Theta(\gamma_+).
\end{equation}
From \cite[Lemma~3]{HofMat12} we  derive that $E(x_\gamma^\delta,\xdag) \le \hat C \, \delta^2/\gamma$ for all $\gamma_0 \le \gamma \le \gamma_{apri}$ with $\hat C=\frac{17}{2 \lambda}$, which yields
$E(x_\gamma^\delta,\xdag) \le \Theta(\gamma)$ for all $\gamma_0 \le \gamma \le \gamma_{apri}$ and hence $\gamma_+ \ge \gamma_{apri}$. Therefore, we have for all $\gamma_0 \le \gamma \le \gamma_+$ the estimate
$$E(x_\gamma^\delta,x_{\gamma_+}^\delta) \le C_E\left(E(x_\gamma^\delta,\xdag)+E(\xdag,x_{\gamma_+}^\delta)\right)
\le C_E(\Theta(\gamma)+\Theta(\gamma_+)) \le 2C_E \Theta(\gamma). $$
This ensures the inequality $\gamma_{\scriptscriptstyle LEP} \ge \gamma_+$. Then we can find
\begin{align*}
E(x_{\gamma_{\scriptscriptstyle LEP}}^\delta,\xdag) &\le C_E\left(E(x_{\gamma_{\scriptscriptstyle LEP}}^\delta,x_{\gamma_+}^\delta)+E(x_{\gamma_+}^\delta,\xdag) \right)  \le 2 C_E^2 \Theta(\gamma_+)+C_E \Theta(\gamma_+) \\
&=C_E(2+C_E)\Theta(\gamma_+),
\end{align*}
which gives (\ref{eq:help}).  By (\ref{eq:help}) we obtain in analogy to the proof of Theorem~3 in \cite{HofMat12} that
\begin{align*}
E(x_{\gamma_{\scriptscriptstyle LEP}}^\delta,\xdag) &\le C_E(2+C_E)\Theta(\gamma_+) \le C_E(2+C_E)\,\Theta(\gamma_k)
=\frac{C_E(2+C_E)}{q} \,\Theta\left(\frac{\gamma_k}{q}\right) \\
&\le \frac{C_E(2+C_E)}{q} \,\Theta\left(\gamma_{apri}\right)   \le C \, C_E\,(2+C_E)\,\frac{17}{2\lambda q
}\,g(\delta),
\end{align*}
which completes the proof.
\end{proof}
\end{itemize}

The following convergence rate estimate follows directly
from \cite[Theorem~4.24]{Flemmingbuch12}, \cite[Theorems~2 and 3]{HofMat12} and Lemma~\ref{lem:Lep}
for the three aforementioned a posteriori choices of regularization parameters.
\begin{proposition} \label{pro:general}
If the variational inequality (\ref{eq:vi})
is valid for a nonnegative error measure $E: \domain{\mathcal{R}} \times \domain{\mathcal{R}} \to \R_+$ with constants $0<\lambda \le 1, \;C>0,$ and a concave index function $g$, then we have
the estimate (\ref{eq:concaverate}) of the convergence rate
for the regularized solutions $x^\delta_{\gamma_*}$ if the regularization parameter $\gamma_*$ is chosen as $\gamma_*=\gamma_{\scriptscriptstyle TDP}$ from the two-sided discrepancy principle,
as $\gamma_*=\gamma_{\scriptscriptstyle SDP}$ from the sequential discrepancy, or  as $\gamma_*=\gamma_{\scriptscriptstyle LEP}$ from the Lepski{\u \i} principle
provided that $E$ satisfies (\ref{eq:E1Lep}) and (\ref{eq:E2Lep}).
\end{proposition}

In the subsequent two sections we will apply Proposition~\ref{pro:general} for
the penalty functionals $\mathcal{R}(x)=\|x\|_{\1}$ and $\mathcal{R}(x)= \eta\|x\|_{\1}+\frac{1}{2}\|x\|_{\2}^2$
(with an arbitrarily fixed $\eta>0$), respectively.

\section{Application to $\ell^1$-regularization}\label{s4}
\setcounter{equation}{0}
\setcounter{theorem}{0}

For $Z:=\2,\;\mathcal{R}(x):=\|x\|_{\1}$ and $\domain{\mathcal{R}}=X=\1$, Proposition~\ref{pro:general} applies with $\lambda=1$, $C=2$ and $g=\varphi$
defined by \begin{equation} \label{eq:varphi}
\varphi(t):=\inf_{n\in\mathbb{N}}\left(\sum_{k=n+1}^\infty\vert x^\dagger_k\vert
+t\sum_{k=1}^n\|f^{(k)}\|_{Y}\right),
\end{equation}
as a consequence of \cite[Theorem~5.2]{BFH13}, in which it was proven that the variational inequality (\ref{eq:vi}) holds  with $E(x,\xdag)=\|x-\xdag\|_{\1}$. The error measure $E$ is a metric in $\domain{\mathcal{R}}$, and hence (\ref{eq:E1Lep}) and
(\ref{eq:E2Lep})  are valid with $C_E=1$.
Thus, we have
\begin{equation}\label{eq:vil1}
\|x-\xdag\|_{\1} \le \|x\|_{\1}-\|\xdag\|_{\1}+ 2\,\varphi\left(\|A(x-\xdag)\|_Y\right)  \quad \mbox{for all} \quad x \in X=\1,
\end{equation}
where $\varphi$ is a concave index function.
Here, the rate function $g$ in (\ref{eq:concaverate})
depends on the decay rate of the remaining components $\xdag_k \to 0$ as $k \to \infty$ and the behaviour
of $\|f^{(k)}\|_{Y}$ (see Assumption~\ref{ass:basic} (c)), which is mostly a growth to infinity as $k \to \infty$. The studies in \cite{HofMat12} ensured that the same rate result is valid for all three  regularization parameter choices $\gamma_*=\gamma_{\scriptscriptstyle TDP},\,\gamma_*=\gamma_{\scriptscriptstyle SDP}$ and $\gamma_*=\gamma_{\scriptscriptstyle LEP}$.

Example~5.3 in \cite{BFH13} makes the convergence rate $g$ in $\ell^1$-regularization explicit as a H\"older rate:
\begin{equation} \label{eq:hoelder}
\|x_{\gamma_*}^\delta-\xdag\|_{\1} = \mathcal{O} \left(\delta^\frac{\mu}{\mu+\nu}\right) \qquad \mbox{as} \qquad\delta \to 0
\end{equation}
for the case when monomials
\begin{equation} \label{eq:power}
\sum \limits_{k=n+1}^\infty |\xdag_k|\le K_1\, n^{-\mu},  \qquad \sum \limits_{k=1}^n \|f^{(k)}\|_{Y} \le K_2\, n^\nu,
\end{equation}
with  exponents $\mu,\nu>0$ and some constants $K_1,K_2>0$, characterize the decay of the solution components and the growth of the $f^{(k)}$-norms, respectively.

On the other hand, Example~5.3 in \cite{BoHo13} outlines the situation that the $f^{(k)}$-norm growth is of
power type, but instead of (\ref{eq:power}) the decay of $\xdag_k \to 0$ is much faster, expressed by an exponential decay rate. Precisely, if
\begin{equation} \label{eq:expo}
\sum \limits_{k=n+1}^\infty |\xdag_k|\le K_1\, \exp(-n^\sigma) \quad \mbox{and}  \quad \sum \limits_{k=1}^n \|f^{(k)}\|_{Y} \le K_2\, n^\nu
\end{equation}
hold with exponents $\sigma,\nu>0$ and some constants $K_1,K_2>0$, then  we have the convergence rate
\begin{equation} \label{eq:log}
\|x_{\gamma_*}^\delta-\xdag\|_{\1} = \mathcal{O} \left(\delta \left(\log(\frac{1}{\delta})\right)^\frac{\nu}{\sigma}\right) \qquad \mbox{as} \qquad\delta \to 0 \,,
\end{equation} instead of (\ref{eq:hoelder}).

The rate (\ref{eq:log}) is not  far from the best rate
\begin{equation} \label{eq:best}
\|x_{\gamma_*}^\delta-\xdag\|_{\1} = \mathcal{O} \left(\delta \right) \qquad \mbox{as} \qquad\delta \to 0,
\end{equation}
which was already established for truly sparse solutions $\xdag \in \ell^0$  in \cite{GrasHaltSch08}. We note that, for all
$\xdag \in \ell^0$ with some $k_{max}$ as the largest index $k \in \N$ such that $\xdag_k \not=0$, there is a uniform constant $K=\sum \limits_{k=1}^{k_{max}} \|f^{(k)}\|_{Y}$ such that the function $\varphi(t)$ in (\ref{eq:varphi}) can be estimated from above by $Kt$. Since mostly $\|f^{(k)}\|_{Y}$ grows rapidly to infinity as $k \to \infty$, the constant $K$ may be large even if $k_{max}$ is not big.
As already mentioned in \cite{JLS09}, the super-rate (\ref{eq:best}) of the $\ell^1$-regularization may not be expected
when the corresponding constant $K$ explodes.

\section{Application to elastic-net regularization}\label{s5}
\setcounter{equation}{0}
\setcounter{theorem}{0}

In this section, we fix the parameter $\eta>0$ arbitrarily and consider the case with
$Z:=\2,\;\mathcal{R}(x):=\mathcal{R}_\eta(x)= \eta\|x\|_{\1}+\frac{1}{2}\|x\|_{\2}^2$ and $\domain{\mathcal{R}_\eta}=X=\1$.
Evidently, the penalty functional $\mathcal{R}_\eta$
is convex, lower semi-continuous and stabilizing in $Z$, which ensures the existence and stability of regularized
solutions $x_{\beta_*,\eta}^\delta$ and its convergence
$\lim \limits_{\delta \to 0} \|x_{\beta_*,\eta}^\delta-\xdag\|_{\2}=0$
for any $\eta>0$ if $\beta_*=\beta_{\scriptscriptstyle TDP}$  or $\beta_*=\beta_{\scriptscriptstyle SDP}$. We refer to \cite{AnzHofMat13,BFH13,JLS09} for further discussions. On the other hand, the application of Proposition~\ref{pro:general} to the elastic-net regularization (\ref{eq:Tiketa}), where
$\beta$ is chosen by TDP and SDP, or LEP,
requires us to construct an appropriate variational inequality.  Below (see Theorem \ref{thm:main}) we will perform this
construction by a weighted superposition of the variational inequality (\ref{eq:vil1}) used in section~\ref{s4}
and a corresponding one for the $\2$-term in the penalty $\mathcal{R}_\eta$, { which will be the main
purpose of Lemma \ref{lemma:l2}.}

 {
Let us recall the
distance function $d^{A^*}_{\xdag}(\cdot)$ defined in (\ref{eq:distfct}), and introduce a continuous and
strictly decreasing auxiliary function for $R>0$ and $\xdag \in \2 \setminus \range(A^*)$:
\begin{equation}\label{eq:phi}
\Phi(R):=[d^{A^*}_{\xdag}(R)]^2/R\,.
\end{equation}
Using the limit conditions
\begin{equation}\label{eq:limit}
 \lim \limits_{R \to 0} \Phi(R)=+\infty \quad \mbox{and} \quad  \lim \limits_{R \to \infty} \Phi(R)=0,
\end{equation}
it is not difficult to see that the function
\begin{equation}\label{eq:psi}
\widehat \psi(t):=[d^{A^*}_{\xdag}(\Phi^{-1}(t))]^2, \quad t>0,
\end{equation}
is an index function. On the other hand, for $\xdag\in \range(A^*)$ we can always find
$R_0>0$ such that $d^{*}_{\xdag}(R_0)=0$.
}

\begin{lemma}\label{lemma:l2}
  {
The variational inequality (\ref{eq:vil1}) holds true for $Z:=\2,\;\mathcal{R}(x):=\frac{1}{2}\|x\|_{\2}$ and $\domain{\mathcal{R}}=X=\2$, $E(x,\xdag):=\|x-\xdag\|_{\2}^2$ and some index functions $g:=\psi$, and positive constants $\lambda$ and $C$.
More precisely,  if $\xdag \in \2 \setminus \range(A^*)$,
we can take $\psi$ as a concave index function $\psi:[0,\infty)\to \R$  such that $\psi(t)\geq \widehat \psi(t) $ for all $t>0$, $\lambda:=\frac{1}{4}$ and $C:=2$; if $\xdag \in  \range(A^*)$, we can set $\psi(t):=t$, $\lambda:=\frac{1}{2}$ and $C:=R_0$ with $R_0>0$ satisfying $d^{A^*}_{\xdag}(R_0)=0$. }
\end{lemma}

\begin{proof}

It is readily checked that
\begin{equation}\label{eq:ell2}
\frac{1}{2}\|x-\xdag\|^2_{\2}= \frac{1}{2}\|x\|^2_{\2}- \frac{1}{2}\|\xdag\|^2_{\2} - \langle \xdag, x-\xdag
\rangle_{\2}.
\end{equation}
{Now we separate two cases.
In the first case, characterized by $\xdag \notin \range(A^*)$, we decompose for arbitrary fixed $R>0$ the element $\xdag$ as $\xdag=A^*\,v_R+u_R$, where
$\|v_R\|_{Y}=R$ and $\|u_R\|_{\2} = d^{A^*}_{\xdag}(R)>0$. It is known (cf.~\cite[p.377-78]{BoHo10}) that
the infimum in (\ref{eq:distfct}) is a minimum in this case
and for all $R>0$, and that such elements
$v_R \in Y$ and $u_R \in \2$ always exist.  Using this fact, we can conclude that
$$- \langle \xdag, x-\xdag
\rangle_{\2} = \langle A^*\,v_R,x-\xdag\rangle_{\2}+\langle u_R,x-\xdag\rangle_{\2}=\langle v_R,A(x-\xdag)\rangle_{Y}+\langle u_R,x-\xdag\rangle_{\2},$$
which yields the estimate
$$ - \langle \xdag, x-\xdag\rangle_{\2} \le R \|A(x-\xdag)\|_Y+d^{A^*}_{\xdag}(R)\|x-\xdag\|_{\2}$$
for the third term in the right-hand side of the identity (\ref{eq:ell2}). Hence, we may employ for \linebreak $\xdag \in \2 \setminus \range(A^*)$
the auxiliary function
$\Phi$, defined by  (\ref{eq:phi}). Thanks to the limit conditions (\ref{eq:limit}),
we can choose $R:=\Phi^{-1}(\|A(x-\xdag)\|_Y)$. Then we obtain upon
Young's inequality that
\begin{align*}
- \langle \xdag, x-\xdag
\rangle_{\2}&\le R \|A(x-\xdag)\|_Y+ [d^{A^*}_{\xdag}(R)]^2+\frac{\|x-\xdag\|_{\2}^2}{4}\\\
&=2\,[d^{A^*}_{\xdag}(\Phi^{-1}(\|A(x-\xdag)\|_Y))]^2+\frac{\|x-\xdag\|_{\2}^2}{4}.
\end{align*}
Consequently, recalling (\ref{eq:ell2}) we know that for $\xdag \in \2 \setminus \range(A^*)$ the variational inequality
\begin{equation}\label{eq:vil2}
0 \le\frac{1}{4} \|x-\xdag\|^2_{\2} \le \frac{1}{2} \|x\|^2_{\2}-\frac{1}{2} \|\xdag\|^2_{\2}+
 2\,[d^{A^*}_{\xdag}(\Phi^{-1}(\|A(x-\xdag)\|_Y))]^2
\end{equation}
is valid for all $x \in \2$.} { Recalling the definition of $\widehat \psi$ and
$\psi$, we know that the last term in the right hand-side  of the inequality (\ref{eq:vil2}) is  exactly $2 \widehat \psi(\|A(x-\xdag)\|_Y))$,
and this inequality holds still true if this term is replaced by $2\psi(\|A(x-\xdag)\|_Y)$.
 }

{In the second case, where a source condition $\xdag \in \range(A^*)$ is satisfied, we know
that there exists some  $R_0>0$ such that $d^{A^*}_{\xdag}(R_0)=0$, and we can simply estimate the third term in the
right-hand side of (\ref{eq:ell2})  as
$$- \langle \xdag, x-\xdag\rangle_{\2} \le R_0  \|A(x-\xdag)\|_Y.$$
Then the variational inequality attains the simpler form
\begin{equation}\label{eq:vil20}
0 \le\frac{1}{2} \|x-\xdag\|^2_{\2} \le \frac{1}{2} \|x\|^2_{\2}-\frac{1}{2} \|\xdag\|^2_{\2}+
R_0\,\|A(x-\xdag)\|_Y\,,
\end{equation}}
{ thus we can set
$\psi(t):=t$.
}
\end{proof}

\begin{theorem} \label{thm:main}
{
Let
$E_\eta(x,\xdag):=\eta \,\|x-\xdag\|_{\1} +\frac{1}{4}\|x-\xdag\|^2_{\2}$
be an error functional,
and $g_\eta$ be a concave index function given by
 \begin{equation}\label{eq:geta}
g_\eta(t):=2\eta\,\varphi(t)+K\psi(t), \quad t>0,
\end{equation}
where $\varphi$ is from (\ref{eq:varphi}) and $\psi$ is defined in Lemma \ref{lemma:l2},
and the constants $K=2$ for $\xdag \in \1 \setminus \range(A^*)$ and
$K=R_0$ for $\xdag \in \1 \cap \range(A^*)$ satisfying $d^{*}_{\xdag}(R_0)=0$.
Then the following variational inequality
\begin{equation} \label{eq:viref}
E_\eta(x,\xdag) \le \mathcal{R}_\eta(x) - \mathcal{R}_\eta(\xdag) + g_\eta(\|A(x-\xdag)\|_Y)\qquad \forall\,x \in \1\,
\end{equation}
holds for the elastic-net regularization (\ref{eq:Tiketa})-(\ref{eq:peneta}) and an arbitrary weight parameter $\eta>0$.
}
%
{
Moreover, as an immediate consequence of the inequality (\ref{eq:viref}), the convergence rate
\begin{equation} \label{eq:thmrate}
E_\eta(x^\delta_{\beta_*,\eta},\xdag)=\mathcal{O}(g_\eta(\delta)) \quad \mbox{as} \quad \delta \to 0
\end{equation}
holds for the elastic-net regularized solutions $x^\delta_{\beta_*,\eta}$ when
the regularization parameter $\beta_*$ is chosen from {TDP, SDP or LEP}.  Thus the following
alternative rate estimate follows
\begin{equation} \label{eq:rates12}
\|x^\delta_{\beta_*,\eta}-\xdag\|_{\2} \le \|x^\delta_{\beta_*,\eta}-\xdag\|_{\1}=\mathcal{O}(\varphi(\delta)+\psi(\delta)) \quad \mbox{as} \quad \delta \to 0.
\end{equation}
}\end{theorem}
\begin{proof}
{Taking into account (\ref{eq:vil2}) and (\ref{eq:vil20}), we can deduce from (\ref{eq:vil1})
the variational inequality
\begin{equation}\label{eq:vieta}
E_\eta(x,\xdag)\le \mathcal{R}_\eta(x) - \mathcal{R}_\eta(\xdag) +  \left[2\,\eta\,\varphi+ K\,\psi\right]
(\|A(x-\xdag)\|_Y)
\end{equation}
 of type (\ref{eq:vi}) for all $x \in \1$,  with the constants $K=2$ for $\xdag \in \1 \setminus \range(A^*)$ and
$K=R_0$ for $\xdag \in \1 \cap \range(A^*)$ satisfying $d^{*}_{\xdag}(R_0)=0$. We also note that the error functional $E_\eta$
satisfies the conditions (\ref{eq:E1Lep}) and (\ref{eq:E2Lep}) with $C_E=2$, and that the function $g_\eta(t)$ is a concave index function.}

Then we can apply Proposition~\ref{pro:general}
with $\lambda=1$, $E=E_\eta$, $g=g_\eta$ and $C=1$
for all three a posteriori parameter choices of $\beta$ under consideration to obtain the desired results.
\end{proof}

\begin{example}\label{ex:Hoelder}
{\rm We discuss now the convergence rate in (\ref{eq:rates12}) under the assumptions that $\xdag$ satisfies the H\"older source condition with
exponent $\theta>0$,
implying a power-type decay of the distance function (\ref{eq:powerdist}), and that
 the power-type decay of solution components and power-type growth of the $f^{(k)}$-norms (\ref{eq:power}) hold.
 Note that $\xdag \in \ell^0$
implies $\xdag \in \range(A^*)$, which is a direct consequence of Assumption~\ref{ass:basic} (c).
Then (\ref{eq:hoelder}) and (\ref{eq:best})  show that  $\varphi(\delta) \sim \delta^\frac{\mu}{\mu+\nu}$ if $\xdag \in
\1 \setminus \ell^0$ and  $\varphi(\delta) \sim \delta$ if $\xdag \in \ell^0$.
On the other hand, the concave index function $\widehat \psi(\delta)=\psi(\delta) \sim \delta^\frac{2\theta}{\theta+1}$ can be seen from formula (\ref{eq:psi}) if $\xdag \in \1 \setminus \range(A^*)$, while it occurs that $\psi(\delta)\sim \delta$
if $\xdag \in \1 \cap \range(A^*)$.
{In summary, we have by Theorem~\ref{thm:main} the H\"older convergence rate
\begin{equation} \label{eq:karate}
\|x^\delta_{\beta_*,\eta}-\xdag\|_{\1}=\mathcal{O}(\delta^\kappa),
\end{equation}
where $\kappa$ is given by
\begin{equation} \label{eq:kap}
\kappa= \left\{ \begin{array}{ccl} 1 & \mbox{if} & \xdag \in \ell^0,\\
\frac{\mu}{\mu+\nu} &\mbox{if} & \xdag \in (\1 \cap \range(A^*))\setminus \ell^0,
\\
\min\left( \frac{\mu}{\mu+\nu},\frac{2\theta}{\theta+1}\right) &\mbox{if} & \xdag \in \ell^1 \setminus \range(A^*).  \end{array} \right.
\end{equation}
In the case $\xdag \in \1 \setminus \ell^0$, we observe
the behaviour of the exponent (\ref{eq:kap}) for the H\"older convergence rate (\ref{eq:karate}) that
fast convergence rates occur only for almost sparse solutions. There is a trade-off here in the sense that the closer
the exponent $\kappa$ is to one the more drastic must be the decay of the components $\xdag_k$
of the solution $\xdag$ if $k$ tends to infinity.
More precisely, for $\kappa$ close to one, the decay exponent $\mu$ from (\ref{eq:power}) for the solution components has to be sufficiently large and the exponent $\theta$ of the range-type source condition occurring
in (\ref{eq:powerdist}) has to be sufficiently close to one.}
}\end{example}

\begin{remark} \label{rem:errorestimate}
{\rm Taking into account the estimates from \cite{HofMat12}, we can distinguish upper bounds of $E_\eta(x^\delta_{\beta_*,\eta},\xdag)$  in Theorem~\ref{thm:main} for $\beta_*=\beta_{\scriptscriptstyle TDP}$ and
$\beta_*=\beta_{\scriptscriptstyle SDP}$, which, however, yield the same convergence rate (\ref{eq:thmrate}).  To be more precise, we find for sufficiently small $\delta>0$ that
\begin{equation} \label{eq:DP}
\eta \|x^\delta_{\beta_*,\eta}-\xdag\|_{\1}+\frac{1}{4}\|x^\delta_{\beta_*,\eta}-\xdag\|^2_{\2} \le C_*\, (2\eta\,\varphi(\delta)+K\psi(\delta))
\end{equation}
holds with
\begin{equation} \label{eq:C*}
C_*= \left\{\begin{array}{ccc} \tau_2+1 & \mbox{for} & \beta_*=\beta_{\scriptscriptstyle TDP}, \\ (\tau+1)\max\{\frac{2(\tau^2+1)}{q(\tau-1)^2(\tau+1)},1\} & \mbox{for} & \beta_*=\beta_{\scriptscriptstyle SDP.}\end{array}  \right.
\end{equation}
From (\ref{eq:DP}) we derive the $\1$-norm estimate
\begin{equation} \label{eq:DP1}
\|x^\delta_{\beta_*,\eta}-\xdag\|_{\1} \le 2 C_*\,\varphi(\delta)+ \frac{C_*}{\eta} K\psi(\delta).
\end{equation}
In contrast to the approach for (\ref{eq:rates12}), we may derive directly from (\ref{eq:DP}) an $\2$-norm estimate of the form
\begin{equation} \label{eq:DP2}
\|x^\delta_{\beta_*,\eta}-\xdag\|_{\2} \le 2\, \sqrt{ C_*\, (2\eta\,\varphi(\delta)+K\psi(\delta))}
\end{equation}
with a lower (square root) rate as $\delta \to 0$. On the other hand,
{it follows from the formula (\ref{eq:estilem})
that the estimate (\ref{eq:DP}) holds true for $\beta_*=\beta_{\scriptscriptstyle LEP}$ with constant
$ 
C_*= \; {34}/{q}.\quad 
$ 
}}\end{remark}

\begin{remark} \label{rem:role_of_weight}{\rm
By introducing the weight parameter $\eta>0$ in section~\ref{s1}, the natural two-parameter regularization (\ref{eq:Tikelastic}) of the elastic-net approach reduces to the one-parameter regularization (\ref{eq:Tiketa}).
If the weight $\eta$ is fixed for all $\delta>0$, then the convergence rate in (\ref{eq:rates12}), also
the H\"older rate expressed by the exponents $\kappa$ in Example~\ref{ex:Hoelder}, is the same for all $0<\eta<\infty$,
but the upper bounds on the right-hand side of (\ref{eq:DP1}) and (\ref{eq:DP2}) depend on $\eta$. A very illustrative situation occurs when we consider TDP with $\tau:=\tau_1=\tau_2$. At least for sufficiently small $\delta>0$, the regularization parameter $\beta_*(\eta)$ is  well-defined for all $\eta>0$ and the pairs $(\beta_*(\eta),\eta)$ form a `discrepancy curve' with $\|Ax^\delta_{\beta_*(\eta),\eta}-\yd\|_Y=\tau \delta$ and the uniform
 convergence rates of all associated regularized solutions. Then we can select one pair from the curve with the goal to implement
additional solution features; see more discussions in \cite[page 166]{LuPer13}.
}\end{remark}

\section{Conclusions}
In this work we have derived some variational inequalities for both  $\1$- and elastic-net regularizations.
Then we have applied these variational inequalities
to obtain some explicit convergence rates, and compared the results with the ones from the classical source conditions.
This
increases significantly the range of the regularized solutions, for which the convergence rates can be achieved.
Three different principles of a posteriori parameter choices are also discussed, and their influences on convergence rates
are analyzed. The basic principles, analysis tools, and the selection strategies for the choice of
regularization parameters can
be equally applied to general multi-parameter Tikhonov-type regularizations.

{
\section{Acknowledgements}

The authors would like to thank the anonymous referees for their many insightful and constructive
suggestions and comments, which have helped us {to} improve the presentation and the results of
the paper significantly.
}

\end{document}